\documentclass[12pt]{amsart}

\usepackage{
amsfonts,
latexsym,
amssymb,
amsbsy,
enumerate,
mathrsfs,
}

\newcommand{\labbel}{\label}

\newtheorem{theorem}{Theorem}%[section]

\newtheorem{proposition}[theorem]{Proposition} 
 
\newtheorem{corollary}[theorem]{Corollary}

\newtheorem*{theorem*}{Theorem}
\newtheorem*{corollary*}{Corollary}
\newtheorem*{proposition*}{Proposition}

\theoremstyle{definition}
\newtheorem{definition}[theorem]{Definition}

\newtheorem*{definition*}{Definition}
\newtheorem*{disclaimer*}{Disclaimer}

\theoremstyle{remark}

\newcommand{\brfrt}{\hspace{0 pt}}

\DeclareMathAlphabet{\mathpzc}{OT1}{pzc}{m}{it}
\DeclareMathOperator{\Spec}{\mathbf{Spec}}
\DeclareMathOperator{\PSpec}{\mathbf{PSpec}}

\begin{document}
 
\title{Generalized Frol\'\i k classes}

\author{Paolo Lipparini} 
\address{Dipartimento di Matematica\\Viale della Ricerca Scientifrol\'\i ka\\II Universit\`a di Roma (Tor Vergata)\\I-00133 ROME ITALY}
\urladdr{http://www.mat.uniroma2.it/\textasciitilde lipparin}


\thanks{Work performed under the aruspices of G.N.S.A.G.A}

\begin{abstract}
The  class 
$\mathfrak C $ relative to 
countably compact topological spaces and the class $\mathfrak P$
relative to  pseudocompact spaces
introduced by Z. Frol\'\i k  
are naturally generalized relative
to every topological property.
We provide a characterization
of such generalized Frol\'\i k  classes
in the broad case of properties
defined in terms of filter convergence.

If a class of spaces
can be defined in terms of
filter convergence, then the same is true for its  Frol\'\i k  class.
\end{abstract} 
 
\maketitle

Zden\v ek Frol\'\i k  \cite{Fr1} introduced   the class $\mathfrak C $ 
of all topological spaces $X$ such that $X \times Y$
is countably compact, for every countably compact $Y$, 
and \cite{Fr2}   the class $\mathfrak P$
of all topological spaces $X$ such that $X \times Y$
is pseudocompact, for every pseudocompact $Y$.
Hence it is natural to define, for every class $\mathcal K$
of topological spaces, the \emph{(generalized) $\mathcal K$-Frol\'\i k class}
$\mathfrak F(\mathcal K)$,
consisting 
of all  topological spaces $X$ such that $X \times Y \in \mathcal K $,
whenever  $Y \in \mathcal K$.
If $X \in \mathfrak F(\mathcal K)$, we shall sometimes simply say that 
\emph{$X$ is $\mathcal K$-Frol\'\i k} and, when convenient and  with some sloppiness,
we shall identify a topological property with the class of spaces satisfying it.
Throughout the present note, ``space'' will be synonymous  of 
``topological space''. 

Notice that Frol\'\i k  \cite{Fr1,Fr2} assumed some separation axiom in the definitions 
of $\mathfrak C $ and $\mathfrak P $; however, in the present note,
no separation axiom is necessary. Of course, every definition  given here
can  be considered relative only to those spaces satisfying 
some separation axiom. If the separation axiom 
in question is preserved under products, then
all theorems proved here remain valid.
See \cite{Vf1,Vf2} for results  and variations
 on  $\mathfrak C $ and $\mathfrak P $
without assuming separation axioms.

Classes analogous to $\mathfrak C $  and  $\mathfrak P$ have subsequently been introduced by other authors.
For example, in  the above terminology, the famous and now solved
Morita's first
conjecture
asserts that   $\mathfrak F$(normal) is the class of all discrete spaces.
See Atsuji \cite{A} and
Balogh \cite{Bal} for more details  and further references.
Another class that is largely studied is
the class of 
 \emph{productively
Lindel\"of spaces}, $\mathfrak F$(Lindel\"of) in the present terminology.
See, e.~g.,   Burton and Tall \cite{BT}. Notice that the ``class 
of spaces considered by Frol\'\i k''  discussed in Burton and Tall's paper
refers to still another class of spaces introduced by Frol\'\i k, seemingly unrelated to the
classes $\mathfrak C $ and $\mathfrak P $.

It is useful to introduce also a relative  notion.
If $\mathcal K$ and  $\mathcal H$
are classes of topological spaces,
  let us say that  
\emph{$X$ is $\mathcal K$-Frol\'\i k for $\mathcal H$-spaces}
if 
 $X \times Y \in \mathcal K $,
whenever  $Y \in \mathcal H$.
In particular, 
$\mathcal K$-Frol\'\i k is the same as
$\mathcal K$-Frol\'\i k
for $\mathcal K$-spaces.
Just to exemplify, the second now solved  Morita conjecture,
in the above terminology, asserts  that
metrizable spaces are exactly 
the normal-Frol\'\i k  spaces for normal $P$-spaces. 

\smallskip

We now turn to the kind of topological properties we shall
 consider here
in connection with generalized Frol\'\i k  classes.
Notions related to ultrafilter convergence
have frequently played an important role in the study of products.
See the surveys Garc\'\i a-Ferreira and Kocinac \cite{GFK},
Stephenson \cite{St} and Vaughan   \cite{Va} 
for  details and further references.

In \cite{sssr} we showed that if we extend some definitions by considering 
filters in place of ultrafilters we obtain
genuinely more general notions.
Hence we give here the general definitions;
however, if the reader wants, he or she might always suppose that all the filters
under consideration are ultrafilters, that is, maximal. 

Recall that if $X$ is a topological space 
and $F$ is a filter over some set $I$,
a sequence $( x_i) _{i \in I} $ of elements of 
$X$ is said to \emph{$F$-converge} to  $x \in X$
if $\{ i \in I \mid  x_i \in U\} \in D$,
for every neighborhood $U$ of $x$ in $X$.

\begin{definition} \labbel{spcpn}
\cite{sssr} Suppose that  $I$ is a set and $\mathcal F$ is a family of filters over $I$.
A topological space $X$  is
\emph{sequencewise $\mathcal F$-\brfrt compact} 
if, for every sequence 
$( x_i) _{i \in I} $ 
of elements of $X$, there is 
$F \in \mathcal F$ 
such that  
$( x_i) _{i \in I} $ 
$F$-converges in $X$.
The class of all 
sequencewise $\mathcal F$-\brfrt compact 
spaces shall be denoted by 
$\mathcal F c$.
See \cite{sssr} and references there
for some history about notions
related to sequencewise $\mathcal F$-\brfrt compactness.

If $Y$ is a topological space and $\mathbf y = ( y_i) _{i \in I} $ 
is a sequence of elements of $Y$,
 the \emph{spectrum  of $\mathbf y $ in $Y$}, in symbols,
$Spec(\mathbf y, Y)$, is the set of all filters $F$ over $I$ 
such that $\mathbf y $  $F$-converges in $Y$. 

If $\mathcal H$ is a class of topological spaces, 
the \emph{spectrum of $\mathcal H$ (relative to $I$)}
is the set
$\Spec_I (\mathcal H )=
\{Spec(\mathbf y, Y)  \mid Y \in \mathcal H 
\text{ and } \allowbreak \mathbf y  \allowbreak \text{ is an $I$-indexed}
\allowbreak \text{sequence
of elements of } Y \}$.  
 \end{definition}   

\begin{proposition} \labbel{swpc}
Suppose that  $\mathcal F$ is a set of filters over $I$.

A topological space $X$ is $\mathcal Fc$-Frol\'\i k  
if and only if $X$ is sequencewise $\mathcal F \cap \mathcal G$-\brfrt compact, 
for every  $\mathcal G \in  \Spec_I (\mathcal Fc)$.

More generally,
if  $\mathcal H$  is a class of topological spaces,
then a topological space $X$ is $\mathcal Fc$-Frol\'\i k
for $\mathcal H$-spaces  
if and only if $X$ is sequencewise 
$\mathcal F \cap \mathcal G$-\brfrt compact, 
for every  $\mathcal G \in  \Spec_I (\mathcal H)$.
 \end{proposition} 

\begin{proof}
The first statement is the particular case of the second statement
when $\mathcal H = \mathcal Fc$.
To prove the second statement, note the following chain 
of equivalences.
  \begin{enumerate}    
\item 
$X$ is $\mathcal Fc$-Frol\'\i k for $\mathcal H$-spaces;
\item  
$X  \times Y $ is sequencewise $\mathcal F$-compact, 
for every $Y \in \mathcal H$;
\item  
for every $Y \in \mathcal H$
and every sequence $( z_i) _{i \in I} $ 
in $X  \times Y $, there is 
$F  \in \mathcal F$ 
such that $( z_i) _{i \in I} $
$F$-converges in 
$X  \times Y $;
\item  
for every $Y \in \mathcal H$
and every pair of sequences $( y_i) _{i \in I} $ 
in $Y $ and
$( x_i) _{i \in I} $ 
in $X $, there is 
$F  \in \mathcal F$ 
such that both $( y_i) _{i \in I} $
$F$-converges in 
$Y $
and $( x_i) _{i \in I} $
$F$-converges in 
$X$;
\item  
for every $Y \in \mathcal H$
and every pair of sequences $\mathbf y = ( y_i) _{i \in I} $ 
in $Y $ and
$( x_i) _{i \in I} $ 
in $X $, there is 
$F  \in \mathcal F \cap Spec(\mathbf y, Y)$ 
such that  $( x_i) _{i \in I} $
$F$-converges in 
$X$;
\item
$X$ is sequencewise $\mathcal F \cap Spec(\mathbf y, Y)$-\brfrt compact, 
for every $Y \in \mathcal H$ 
and every sequence $\mathbf y = ( y_i) _{i \in I} $ 
of elements of $Y$.
\item
$X$ is sequencewise 
$\mathcal F \cap \mathcal G$-\brfrt compact, 
for every  $\mathcal G \in  \Spec_I (\mathcal H)$.
  \end{enumerate}

The equivalence of (1) and (2) is the definition of $\mathcal Fc$-Frol\'\i kness for $\mathcal H$ spaces; 
(2) $\Leftrightarrow $  (3)  follows from the definition of
sequencewise $\mathcal F$-\brfrt compactness;
 (3) $\Leftrightarrow $  (4) 
follows from the easy and well-known fact that a
sequence in a product $F$-converges if and only if each component
of the sequence onto any factor $F$-converges; 
(4) $\Leftrightarrow $  (5) is from the definition of 
$Spec(\mathbf y, Y)$; 
(5) $\Leftrightarrow $  (6)
  follows from  the definition 
of sequencewise $\mathcal Q$-\brfrt compactness,
for $\mathcal Q = \mathcal F \cap Spec(\mathbf y,Y)$;
finally, (6) $\Leftrightarrow $   (7)
follows from the definition of   
$\Spec_I (\mathcal H)$.
 \end{proof}  
  
Proposition \ref{swpc}  
shows that the study of 
$\mathcal Fc$-Frol\'\i kness  
can be divided into two steps.
In  the  first step one has to  determine
$\Spec_I (\mathcal Fc)$,
that is, 
 all possible values for
$Spec(\mathbf y, Y)$,
$Y$ varying in 
$\mathcal Fc$,
and $\mathbf y$ varying among all $I$-indexed sequences on $Y$. 
In the 
second step one should  characterize  those spaces which are
sequencewise 
$\mathcal Q$-\brfrt compact
for all  $\mathcal Q$
having the form $\mathcal F \cap \mathcal G$, 
for some  $\mathcal G \in  \Spec_I (\mathcal Fc)$.

Each of the two steps might prove to be very difficult.

Concerning the first step, one could
notice that there are 
very little constraints on the values that $Spec(\mathbf y, Y)$
might assume, in general.
For simplicity, suppose that $\mathcal G$ is a set of ultrafilters 
over $I$, and that $\mathcal G$ contains all principal ultrafilters.
Give $I$ the discrete topology, let $ \beta (I)$
be its \v Cech-Stone compactification and,
as usual, identify $I$ with the set of all principal ultrafilters over $I$.
Thus, $I \subseteq \mathcal G \subseteq \beta (I)$.
Think of $\mathcal G$ as a subspace $Y$ of $\beta(I)$,
with the induced topology.
Letting   
$\mathbf y$ be the sequence which is the identity on $I$,
it is trivial to check that 
$Spec(\mathbf y, Y) = \mathcal G \cup \{ \mathrm P(I)  \} $.
Thus  $Spec(\mathbf y, Y)$ can be quite arbitrary
(of course, it necessarily contains all the principal ultrafilters,
as well as the improper filter).
On the other hand, elaborated constraints arise if we want
 $Y$ to be sequencewise $\mathcal F$-\brfrt compact,
for some $\mathcal F$, 
since, in order to fulfill this request, we need to check that  \emph{for every}
sequence, there is $F \in \mathcal F$  
such that the sequence $F$-converges.

As far as the second step is concerned,
the difficulties of studying
sequencewise $\mathcal F$-\brfrt compactness
have been hinted in \cite[Section 6]{sssr}. 

In any case, Proposition \ref{swpc}  shows
that the study of $\mathcal Fc$-Frol\'\i kness   
can be reduced to the study of 
simultaneous sequencewise $\mathcal Q$-\brfrt compactness,
for a (perhaps large number of) appropriate sets $\mathcal Q$. 
The above observation suggests the following definition.

\begin{definition} \labbel{tanti}
Suppose that  $\mathbf F = \{ \mathcal F_a \mid  a\in A\}$,
where each $\mathcal F_a$ is a family of filters over
some set $I_a$.
We say that a topological space $X$ is
\emph{$\mathbf F$-compact} if 
$X$ is sequencewise $\mathcal F_a$-\brfrt compact,
for every $a \in A$.
    
 The class of all $\mathbf F$-compact spaces
shall be denoted by $\mathbf Fc$. 
 \end{definition}   

$\mathbf F$-compactness is not necessarily equivalent to
sequencewise $\mathcal F$-\brfrt compactness, for some $\mathcal F$.
See \cite{notefc}.

Let us denote by
$(\mathcal K : \mathcal H)$
the class of all spaces which are $\mathcal K$-Frol\'\i k  
for $\mathcal H$-spaces.

\begin{theorem} \labbel{FQ}
For every  
$\mathbf F = \{ \mathcal F_a \mid  a\in A\}$ as above, there is
$\mathbf Q = \{ \mathcal Q_b \mid  b\in B\}$
such that 
$\mathfrak F(\mathbf Fc)= \mathbf Qc$.

More generally, for every  
$\mathbf F = \{ \mathcal F_a \mid  a\in A\}$
and every class $\mathcal H$ of topological spaces, there is
$\mathbf Q = \{ \mathcal Q_b \mid  b\in B\}$
such that 
$(\mathbf Fc : \mathcal H)= \mathbf Qc$.
\end{theorem}

 \begin{proof}
The first statement follows from the second one, by taking
$\mathcal H = \mathbf Fc$, since, by  the definitions, 
$(\mathbf Fc : \mathbf Fc)= \mathfrak F(\mathbf Fc)$.

Let us prove the rest of the theorem. Since, by definition, 
 $\mathbf Fc = \bigcap _{ a\in A}  \mathcal F_ac $,
we have that
$(\mathbf Fc : \mathcal H)
=
\bigcap _{ a\in A} ( \mathcal F_ac : \mathcal H)$,
 by an obvious property of the binary operator $( \text{-} : \text{-})$. 
By the last statement in Proposition \ref{swpc},
for every $a \in A$, 
there are a set $B_a$ 
 and families 
$(\mathcal Q_b) _{b \in B_a} $
such that 
a topological space  $ X$ belongs to 
$( \mathcal F_ac : \mathcal H)$ 
if and only if 
$X$ is  
sequencewise $\mathcal Q_b$-\brfrt compact,
for every  $b \in B_a$. 
Since 
$(\mathbf Fc : \mathcal H)
=
\bigcap _{ a\in A} ( \mathcal F_ac : \mathcal H)
$, we get that a topological space  $X$
belongs to 
$(\mathbf Fc : \mathcal H)$
if and only if $X$ is
$\mathbf Q$-compact,
for $\mathbf Q=  \{ \mathcal Q_b \mid b \in \bigcup _{a \in A}  B_a\}   $.  
 \end{proof}  

The proofs of Proposition \ref{swpc} and of Theorem \ref{FQ} 
give an explicit description of the $\mathbf Q $ given by \ref{FQ}.

\begin{corollary} \labbel{corFQ}
For every  
class $\mathcal H$ of topological spaces
and every 
$\mathbf F $ as in Definition \ref{tanti},
$(\mathbf Fc : \mathcal H)= \mathbf Qc$,
where
$\mathbf Q = \{ \mathcal F \cap \mathcal G \mid  
\mathcal F \in \mathbf F,\ \mathcal G \in \Spec_I (\mathcal H) \}$.

 In particular, the 
 $\mathbf Q $ given by Theorem \ref{FQ} 
can be chosen in such  a way that,
for every $\mathcal Q \in \mathbf Q $,
there is  $\mathcal F \in \mathbf F $
such that 
$\mathcal Q \subseteq \mathcal F$.
If this is the case, and every  
$\mathcal F \in \mathbf F $ is a family of ultrafilters, 
then every
$\mathcal Q \in \mathbf Q $ is a family of ultrafilters.
  \end{corollary}

\begin{corollary} \labbel{flind}
(a) There is $\mathbf Q $ 
such that $\mathfrak F \text{(Lindel\"of)} = \mathbf Qc $.
  
(b) There is $\mathbf Q $ 
such that $\mathfrak F \text{(linearly Lindel\"of)} = \mathbf Qc $.  
 \end{corollary}

 \begin{proof}
(a) The Lindel\"of property can be expressed as the conjunction
of $[ \omega _1, \lambda ]$-\brfrt compactness,  for every cardinal $\lambda  > \omega$,
equivalently, as the conjunction of 
$[ \lambda , \lambda ]$-\brfrt compactness  for every cardinal $\lambda > \omega $.
Recall that a 
space is \emph{$[ \mu, \lambda ]$-\brfrt compact}  
if  every  open cover by at most $\lambda$ sets
has a subcover with $<\mu $ sets.

Ginsburg and Saks \cite{GS}, Saks \cite{Sa} and, in full generality,
 Caicedo \cite{Ca}   showed that $[ \mu, \lambda ]$-\brfrt compactness
is equivalent, in the present terminology, to
sequencewise $\mathcal F$-\brfrt compactness, for an appropriate
$\mathcal F$. 
 The corollary is now immediate from 
Theorem \ref{FQ}.
 
(b) A space is linearly Lindel\"of 
if and only if it is
$[ \lambda , \lambda ]$-\brfrt compact,  for every 
\emph{regular} 
cardinal $\lambda > \omega $.
This can be taken as the definition of linear 
Lindel\"ofness.
The same arguments as in (a)
give the result.
\end{proof}

The families  given by 
Corollary \ref{flind} 
might possibly be  proper classes.
This can be dealt with the usual methods
and causes no set-theoretical problem.
Notice that Corollary \ref{corFQ} 
shows that, in both cases, the families 
$\mathbf Q $ given by 
Corollary \ref{flind}
can be chosen to consist only of ultrafilters.

It is probably interesting, and perhaps useful,
to exemplify Proposition \ref{swpc}  
and Theorem \ref{FQ} 
in other particular cases.
We have not yet worked out any detail.

The above arguments carry over without 
 essential modifications
in order to deal with
pseudocompact-like notions.
Just take into account everywhere sequences of 
nonempty open sets and their $F$-limit points,
in place of sequences of elements and $F$-convergence. 

\begin{definition} \labbel{ps}
Again, suppose that $X$ is a topological space,
$F$ is a filter over some set $I$
and $\mathcal F$ is a family of filters over $I$.

A point  $x$ of $X$ is said to be an
\emph{$F$-limit point} of  a sequence 
$( X_i) _{i \in I} $ of subsets of 
$X$ 
if $\{ i \in I \mid  X_i \cap  U  \not= \emptyset    \} \in D$,
for every neighborhood $U$ of $x$ in $X$. 

The topological space $X$  is
\emph{sequencewise $\mathcal F$-\brfrt pseudocompact} 
\cite{sssr} if for every sequence 
$( O_i) _{i \in I} $ 
of nonempty open subsets of $X$
there is 
$F \in \mathcal F$ 
such that  
$( O_i) _{i \in I} $ 
has an $F$-limit  point in $X$.
See again
\cite{GFK,sssr} for related notions and further references. 
The class of all 
sequencewise $\mathcal F$-\brfrt pseudocompact 
spaces shall be denoted by 
$\mathcal Fp$.

If $Y$ is a topological space and $\mathbf O = ( O_i) _{i \in I} $ 
is a sequence of subsets of $Y$,
 the \emph{pseudospectrum  of $\mathbf O $ in $Y$}, in symbols,
$PSpec(\mathbf O, Y)$, is the set of all filters $F$ over $I$ 
such that $\mathbf O$ has an  $F$-limit point  in $Y$. 

If $\mathcal H$ is a class of topological spaces, 
the \emph{pseudospectrum of $\mathcal H$ (relative to $I$)}
is the set
$\PSpec_I (\mathcal H )=
\{PSpec(\mathbf O, Y)  \mid Y \in \mathcal H 
\text{ and } \allowbreak \mathbf O  \allowbreak \text{ is an } \allowbreak 
I\text{-}\allowbreak\text{index}\allowbreak\text{ed}
\allowbreak \text{ sequence
of nonempty open subsets  of } Y \}$.  
 
If $\mathbf F = \{ \mathcal F_a \mid  a\in A\}$
and each $\mathcal F_a$ is a family of filters over
some set $I_a$,
we say that a topological space $X$ is
\emph{$\mathbf F$-pseudocompact} if 
$X$ is sequencewise $\mathcal F_a$-\brfrt pseudocompact,
for every $a \in A$.    

 The class of all $\mathbf F$-pseudocompact spaces
shall be denoted by $\mathbf Fp$. 
\end{definition}

\begin{theorem} \labbel{prop} 
A 
space $X$ is $\mathcal Fp$-Frol\'\i k
for $\mathcal H$-spaces  
if and only if $X$ is sequencewise 
$\mathcal F \cap \mathcal G$-\brfrt pseudocompact, 
for every  $\mathcal G \in  \PSpec_I \mathcal H$.

More generally, 
for every  
class $\mathcal H$ of topological spaces
and every 
$\mathbf F $ as in Definition \ref{tanti},
$(\mathbf Fp : \mathcal H)= \mathbf Qp$,
where
$\mathbf Q = \{ \mathcal F \cap \mathcal G \mid  
\mathcal F \in \mathbf F,\ \mathcal G \in \PSpec_I (\mathcal H) \}$.

\end{theorem}  

One needs very little structure, in order to define generalized
Frol\'\i k  classes. In fact, a binary operation suffices
and, in the case of semigroups, the Frol\'\i k  operator satisfies some
reasonable properties.

\begin{definition} \labbel{frolbin}
Suppose that $\mathcal C$  is a class and a binary operation
$\times$ is defined on  $\mathcal C$.
If $\mathcal K, \mathcal H  \subseteq \mathcal C$, 
we let 
$(\mathcal K : \mathcal H)=
\{ X \in \mathcal C \mid X \times Y \in \mathcal K,
\text{ for every } Y \in  \mathcal H \} $.
This is sometimes called the \emph{residuation operation}.
We let $\mathfrak F(\mathcal K) = (\mathcal K : \mathcal K)$.

Explicit reference to $\mathcal C$  and $\times$
shall not be indicated; however, as we shall show,
the above notions are highly dependent on the choice of 
 $\mathcal C$.
 \end{definition}   

The next proposition lists some very easy properties
of the operator $\mathfrak F$.
Let us say that a class $\mathcal K \subseteq \mathcal C$ 
is \emph{factor closed} if, for every $X, Y \in \mathcal C$, we have that
$X \times Y \in \mathcal K$ implies $X \in \mathcal K$.
In the particular case of topological spaces, 
if $\mathcal K$ is  closed under surjective images
(and does not contain the empty space),
then $\mathcal K$ is factor closed.

\begin{proposition} \labbel{easy} 
Under the assumptions and notations in Definition \ref{frolbin},
the following hold.
   \begin{enumerate}[(a)]
    \item 
$ \mathfrak F(\mathcal K) \supseteq \mathcal K$
if and only if $\mathcal K$ is closed under $\times$.  
\item 
If $\mathcal K \not= \emptyset $ and $\mathcal K$ is factor closed, 
then
 $\mathfrak F(\mathcal K) \subseteq \mathcal K   $.
\item
If  $\times$ is associative, then
$\mathfrak F(\mathcal K)$ is closed under $\times$ 
and $\mathfrak F(\mathfrak F(\mathcal K) ) \supseteq  \mathfrak F( \mathcal K )  $.
\item
If 
$\mathcal C$ has an identity element, then
 $\mathfrak F(\mathfrak F(\mathcal K) ) \subseteq  \mathfrak F( \mathcal K )  $.
  \end{enumerate}
\end{proposition}

  \begin{proof}
We leave (a) and (b) to the reader.

(c) It is trivial to see that 
if $\times$ is associative, then
$ \mathfrak F( \mathcal K )$ is closed under $\times$, hence
(a) with
$ \mathfrak F( \mathcal K )$ in place of $\mathcal K$  
gives $\mathfrak F(\mathfrak F(\mathcal K) ) \supseteq 
 \mathfrak F( \mathcal K )  $.

(d) Trivially the identity element $E$ 
belongs to $ \mathfrak F(\mathcal K)$.
Hence if
$Z \in  \mathfrak F(\mathfrak F( \mathcal K ))$
then
 $Z =Z \times E \in \mathfrak F(\mathcal K)$.
 \end{proof}

Strictly speaking, the class of topological spaces 
has not an identity element with respect to $\times$.
However, considering the equivalence classes  modulo 
homeomorphisms, the one-element space is indeed an identity
element. Moreover, the product operation
is associative among equivalence classes. Thus if 
$\mathcal K$ is a class of topological spaces and $\mathcal K$
is closed under homeomorphisms, then
from Proposition \ref{easy}(c)(d) we get
$\mathfrak F(\mathfrak F(\mathcal K) )= \mathfrak F( \mathcal K )  $.
 
The operator $\mathfrak F$ is highly dependent on 
the class $\mathcal C$. 
Take $\mathcal C$ to be the class of all topological spaces 
with at least two elements and let 
$\mathcal K$ be the class of those spaces 
which have cardinality in 
the set $\{ 2, 2^5 \} \cup \{ 2 ^{9+2h} \mid h \in \omega  \} $.
Then, relative to $\mathcal C$, the classes, respectively, 
 $\mathfrak F(\mathcal K) $,  
$\mathfrak F(\mathfrak F( \mathcal K )) $, and  
$\mathfrak F(\mathfrak F(\mathfrak F( \mathcal K ))) $  
are the classes of spaces with cardinality in, respectively,
$\{  2^4 \} \cup \{ 2 ^{8+2h} \mid h \in \omega  \} $,
$\{ 2 ^{4+2h} \mid h \in \omega  \} $ and 
$ \{ 2 ^{2+2h} \mid h \in \omega  \} $.
Thus
$\mathfrak F(\mathcal K) \subset
\mathfrak F(\mathfrak F( \mathcal K )) \subset
\mathfrak F(\mathfrak F(\mathfrak F( \mathcal K ))) $, where
the inclusions are strict.
On the other hand, relative to the class of all topological spaces,
we have showed that
$\mathfrak F(\mathcal K) ) =
\mathfrak F(\mathfrak F( \mathcal K )) $.

\begin{disclaimer*}
The author has done his best efforts to compile the following
list of references in the most accurate way;
however he cannot exclude the possibility of
 inaccuracies.
Henceforth the author  strongly discourages the use 
of indicators extracted from the list in decisions about individuals, attributions of funds, selections or evaluations of research projects, etc.
A more detailed disclaimer can be found at the author's web page.
\end{disclaimer*}

\end{document}